\documentclass[12pt,reqno]{amsart}
\usepackage{enumerate, latexsym, amsmath, amsfonts, amssymb, amsthm, graphicx, color}
\def\pmod #1{\ ({\rm{mod}}\ #1)}
\def\Z{\Bbb Z}

\def\Q{\Bbb Q}

\def\bg{\bigg}
\def\({\bg(}
\def\){\bg)}

\def\sgn{{\rm sgn}}

\def\sgn{{\rm sgn}}
\def\Arg{{\rm Arg}}

\def\ve{\varepsilon}

\def\Ack{\medskip\noindent {\bf Acknowledgments}}
\theoremstyle{plain}
\newtheorem{theorem}{Theorem}

\newtheorem{lemma}{Lemma}
\newtheorem{corollary}{Corollary}

\theoremstyle{definition}

\theoremstyle{remark}
\newtheorem{remark}{Remark}

\makeatletter
\@namedef{subjclassname@2010}{%
  \textup{2010} Mathematics Subject Classification}
\makeatother
 \vspace{4mm}

\begin{document}
 \baselineskip=17pt
\hbox{} {}
\medskip
\title[Quadratic residues and related permutations ]
{Quadratic residues and related permutations }
\date{}
\author[Hai-Liang Wu] {Hai-Liang Wu}

\thanks{2010 {\it Mathematics Subject Classification}.
Primary 11A15; Secondary 05A05, 11R18.
\newline\indent {\it Keywords}. quadratic residues, permutations, primitive roots, cyclotomic fields.
\newline \indent Supported by the National Natural Science
Foundation of China (Grant No. 11571162).}

\address {(Hai-Liang Wu) Department of Mathematics, Nanjing
University, Nanjing 210093, People's Republic of China}
\email{{\tt whl.math@smail.nju.edu.cn}}

\begin{abstract}
Let $p$ be an odd prime. For any $p$-adic integer $a$ we let $\overline{a}$ denote the unique integer $x$ with $-p/2<x<p/2$ and $x-a$
divisible by $p$. In this paper we study some permutations involving quadratic residues modulo $p$. For instance, we consider the following three sequences.
\begin{align*}
&A_0: \overline{1^2},\ \overline{2^2},\ \cdots,\ \overline{((p-1)/2)^2},\\
&A_1: \overline{a_1},\ \overline{a_2},\ \cdots,\ \overline{a_{(p-1)/2}},\\
&A_2: \overline{g^2},\ \overline{g^4},\ \cdots,\ \overline{g^{p-1}},
\end{align*}
where $g\in\Z$ is a primitive root modulo $p$ and $1\le a_1<a_2<\cdots<a_{(p-1)/2}\le p-1$ are all quadratic residues modulo $p$.
Obviously $A_i$ is a permutation of $A_j$ and we call this permutation $\sigma_{i,j}$. Sun obtained the sign of $\sigma_{0,1}$ when
$p\equiv 3\pmod4$. In this paper we give the sign of $\sigma_{0,1}$ and determine the sign $\sigma_{0,2}$
when $p\equiv 1\pmod 4$.
\end{abstract}

\maketitle

\section{Introduction}
\setcounter{lemma}{0}
\setcounter{theorem}{0}
\setcounter{corollary}{0}
\setcounter{remark}{0}
\setcounter{equation}{0}
\setcounter{conjecture}{0}
\setcounter{proposition}{0}
Let $S$ be a finite set, and let $\tau$ be a permutation on $S$. Throughout this paper the sign of $\tau$ is denoted by $\sgn(\tau)$.
Investigating the properties of $\tau$ is a classical topic in
number theory and combinatorics. In particular, when $S$ is the finite field $\mathbb{F}_p$ with $p>2$ prime and $\tau$ is induced by
a
permutation polynomial $f(x)$ over $\mathbb{F}_p$, the properties of $\tau$ have been studied deeply by many mathematicians.

For instance, when $f(x)=ax$ with $a\in\Z$ and $\gcd(a,p)=1$, the famous Zolotarev's Theorem \cite{Z} states that the Legendre symbol
$(\frac{a}{p})$ is the sign of the permutation of $\Z/p\Z$ induced by multiplication by $a$. Using this, Zolotarev obtained a new
proof of the law of quadratic reciprocity. In this line, Sun \cite{S} investigated many permutation problems involving quadratic
permutation polynomial over $\mathbb{F}_p$. When $k\in\Z^{+}$ and $\gcd(k,p-1)=1$, it is easy to see that $f(x)=x^k$ is a permutation
polynomial over $\mathbb{F}_p$. L.-Y Wang and the author
\cite{WW} determined the sign of permutation $\tau$ induced by $f(x)$. Moreover, W. Duke and K. Hopkins \cite{DH} extended this topic to
an arbitrary finite group. They also generalized the law of quadratic reciprocity on finite groups.

Given an odd prime $p$. Throughout this paper, for any integer $a$ we let $\overline{a}$ denote the unique integer $x$
with $-p/2<x<p/2$ and $x-a$ divisible by $p$. Let $1=a_1<a_2<\cdots<a_{\frac{p-1}{2}}\le p-1$ be all quadratic residues modulo $p$.
Note that $\overline{1^2},\ \overline{2^2},\ \cdots,\ \overline{((p-1)/2)^2}$ is a permutation of $\overline{a_1},\ \overline{a_2},\ \cdots,\ \overline{a_{(p-1)/2}}$. Sun
first studied this problem and obtained the sign of this permutation in the case of $p\equiv3\pmod4$ by evaluating the product
$$\prod_{1\le i<j\le \frac{p-1}{2}}(\zeta_p^{j^2}-\zeta_p^{i^2}),$$
where $\zeta_p=e^{2\pi i/p}$. Inspired by Sun's work, we consider the following sequences.
\begin{align*}
&A_0: \overline{1^2},\ \overline{2^2},\ \cdots,\ \overline{((p-1)/2)^2},\\
&A_1: \overline{a_1},\ \overline{a_2},\ \cdots,\ \overline{a_{(p-1)/2}},\\
&A_2: \overline{g^2},\ \overline{g^4},\ \cdots,\ \overline{g^{p-1}},
\end{align*}
where $g\in\Z$ is a primitive root modulo $p$. Clearly $A_i$ is a permutation of $A_j$ and we call this permutation $\sigma_{i,j}$. As
mentioned before, when $p\equiv 3\pmod4$, Sun proved that
$$\sgn(\sigma_{0,1})=\begin{cases}1&\mbox{if}\ p\equiv 3\pmod8,\\(-1)^{(h(-p)+1)/2}&\mbox{if}\ p\equiv 7\pmod8,\end{cases}$$
where $h(-p)$ is the class number of $\Q(\sqrt{-p})$.

In this paper, we study the above permutations. We first introduce some notations.
Let $p\equiv 1\pmod 4$ be a prime. Let $\ve_p=(u_p+v_p\sqrt{p})/2>1$ and $h(p)$ be the fundamental unit and the class number of
$\Q(\sqrt{p})$ respectively. And we let $h(-4p)$ denote the class number of $\Q(\sqrt{-p})$.
We also define
$$s_p:=\prod_{\substack{1\le k\le (p-1)/2\\ (\frac{k}{p})=-1}}k$$
and let
$$r_p^*:=\#\{(x,y): 1\le x,y\le (p-1)/2,\ x+y\le (p-1)/2,\ \(\frac{x}{p}\)=\(\frac{y}{p}\)=1\},$$
where $\#S$ denotes the cardinality of a finite set $S$. Now we are in the position to state our first theorem.

\begin{theorem}\label{Thm p=1 mod4}
Let $p\equiv1\pmod4$ be a prime. Then
$$\sgn(\sigma_{0,1})=\begin{cases}\overline{s_pu_p^{(p-1)/4}}\cdot(-1)^{((h(p)+2)(p-1)+2h(-4p))/8+r_p^*}&\mbox{if}\ p\equiv 1\pmod8,
\\\\\overline{\frac12s_pu_p^{(p+3)/4}}\cdot(-1)^{((h(p)+2)(p+3)+2h(-4p)-4)/8+r_p^*}&\mbox{if}\ p\equiv 5\pmod8.\end{cases}$$
\end{theorem}

Now we turn to the permutation $\sigma_{0,2}$. For convenience, we write $p=2n+1$ for some positive integer $n$.

\begin{theorem}\label{Thm primitive root}
Let $p\equiv1\pmod4$ be a prime, and let $g\in\Z$ be a primitive root modulo $p$ Then
$$\sgn(\sigma_{0,2})=\overline{\frac12n^{n/2}u_pg^{(p-1)(3n^2-n-2)/8}}\cdot(-1)^{(h(p)+1)/2}.$$
In particular, when $p\equiv 5\pmod 8$, $\sgn(\sigma_{0,2})$ is independent on the choice of $g$ and we have $$\sgn(\sigma_{0,2})=\overline{\frac12n^{n/2}u_p}\cdot(-1)^{(2h(p)+n)/4}.$$
\end{theorem}

As an application of Theorem \ref{Thm primitive root}, we can calculate some determinants concerning Dirichlet characters modulo $p$.
Let $p$ be an odd prime and let $g$ be a primitive root modulo $p$. Let $\widehat{\mathbb{F}_p}$ denote the group of Dirichlet
characters modulo $p$ and let $\chi$ be a generator of $\widehat{\mathbb{F}_p}$. We consider the following matrix:
$$ M_p=\left( \begin{matrix} \chi(1^2) & \chi(2^2) & \cdots & \chi((\frac{p-1}{2})^2) \\ \chi^2(1^2) & \chi^2(2^2) & \cdots &
\chi^2((\frac{p-1}{2})^2) \\ \vdots & \vdots & \ddots & \vdots \\ \chi^{\frac{p-1}{2}}(1^2) & \chi^{\frac{p-1}{2}}(2^2) & \cdots &
\chi^{\frac{p-1}{2}}((\frac{p-1}{2})^2) \\ \end{matrix} \right).$$
We have the following result.

\begin{corollary}\label{determinant}
Let $p\equiv 5\pmod8$ be a prime. Then
$$\det(M_p)=n^{n/2}(-1)^{(n+2)/4}\times\(\overline{\frac12n^{n/2}u_p}\cdot(-1)^{(2h(p)+n)/4}\)\in\Z.$$
\end{corollary}

\begin{remark}
The signs of permutations have deep connections with the calculations of determinants. The readers may see \cite{S19} for more
examples.
\end{remark}

The proofs of Theorem 1.1--1.2 and Corollary 1.1 will be given in the next section.
\medskip

\maketitle
\section{Proofs of Theorem \ref{Thm p=1 mod4}--\ref{Thm primitive root} and Corollary \ref{determinant}}
\setcounter{lemma}{0}
\setcounter{theorem}{0}
\setcounter{corollary}{0}
\setcounter{remark}{0}
\setcounter{equation}{0}
\setcounter{conjecture}{0}

Let $p$ be an odd prime. Throughout this section, we set $p=2n+1$.
We begin with the following lemma (cf. \cite[(1.5)]{S}).

\begin{lemma}\label{j^2-i^2}
$$\prod_{1\le i<j\le n}(j^2-i^2)\equiv\begin{cases}-n!\pmod p&\mbox{if}\ p\equiv 1\pmod4,\\1\pmod p&\mbox{if}\ p\equiv
3\pmod4.\end{cases}$$
\end{lemma}

Mordell \cite{M} showed that if $p>3$ is a prime and $p\equiv 3\pmod4$ then
$$\(\frac{p-1}{2}\)!\equiv (-1)^{\frac{h(-p)+1}{2}}\pmod p,$$
where $h(-p)$ is the class number of $\Q(\sqrt{-p})$. Later S. Chowla \cite{C} extended Mordell's result and obtained the following
result.

\begin{lemma}\label{Chowla lemma}
Let $p\equiv1\pmod4$ be a prime. Then we have
$$\(\frac{p-1}{2}\)!\equiv \frac{(-1)^{(h(p)+1)/2}u_p}{2}\pmod p,$$
where $h(p)$ and $u_p$ are defined as in Theorem \ref{Thm p=1 mod4}.
\end{lemma}

In 1982 K. S. Williams and J. D. Currie \cite{WC} obtained the following result.

\begin{lemma}\label{K.S Williams lemma}
Let $p\equiv1\pmod4$ be a prime. We have
$$2^{\frac{p-1}{4}}\equiv\begin{cases}(-1)^{(p-1)/8+h(-4p)/4}\pmod p&\mbox{if}\ p\equiv 1\pmod8,
\\(-1)^{(p-5)/8+(h(-4p)-2))/4}\cdot(\frac{p-1}{2})!\pmod p&\mbox{if}\ p\equiv 5\pmod8.\end{cases}$$

\end{lemma}

\noindent{\it Proof of Theorem}\ 1.1. Let $p\equiv 1\pmod4$ be a prime. It follows from definition that
$$\sgn(\sigma_{0,1})=\prod_{1\le i<j\le n}\frac{\overline{j^2}-\overline{i^2}}{\overline{a_j}-\overline{a_i}}.$$
We first consider the numerator. By Lemma \ref{j^2-i^2}--\ref{Chowla lemma} we have
\begin{equation}\label{equation numerator}
\prod_{1\le i<j\le n}(j^2-i^2)\equiv -n!\equiv -1\cdot\frac{(-1)^{(h(p)+1)/2}u_p}{2}\pmod p.
\end{equation}
We now turn to the denominator. It is clear that
$$\prod_{1\le i<j\le n}(a_j-a_i)=\prod_{1\le k\le p-1}k^{A_k^{++}},$$
where $A_k^{++}:= \#\{1\le x\le p-1-k: (\frac{x}{p})=(\frac{x+k}{p})=1\}$. To calculate $A_k^{++}$, we set
\begin{align*}
A_k^{--}:=& \#\{1\le x\le p-1-k: \(\frac{x}{p}\)=\(\frac{x+k}{p}\)=-1\},\\
A_k^{+-}:=& \#\{1\le x\le p-1-k: \(\frac{x}{p}\)=1,\ \(\frac{x+k}{p}\)=-1\},\\
A_k^{-+}:=& \#\{1\le x\le p-1-k: \(\frac{x}{p}\)=-1,\ \(\frac{x+k}{p}\)=1\}.
\end{align*}
We also let $r_k:=\#\{1\le x\le p-1-k: (\frac{x}{p})=1\}.$ Then clearly we have
$$A_k^{++}+A_k^{+-}=r_k,\ \text{and}\ A_k^{-+}+A_k^{--}=p-1-k-r_k.$$
From the map $x\mapsto p-k-x$, we immediately obtain that
$$A_k^{+-}=A_k^{-+}.$$
Hence $A_k^{++}+A_k^{-+}=r_k.$ Noting that
$$A_k^{++}+A_k^{+-}+A_k^{-+}+A_k^{--}=p-1-k,$$
we obtain that
$$A_k^{+-}+A_k^{--}=p-1-k-r_k.$$
Moreover, we have
$$A_k^{++}+A_k^{--}-A_k^{+-}-A_k^{-+}=\sum_{1\le x\le p-1-k}\(\frac{x}{p}\)\(\frac{x+k}{p}\).$$
Combining the above equations, we obtain
$$A_k^{++}=r_k+\frac{-p+k+1}{4}+\frac{1}{4}\sum_{1\le x\le p-1-k}\(\frac{x^2+kx}{p}\).$$
Replacing $k$ by $p-k$, we get
$$A_{p-k}^{++}=r_{p-k}+\frac{-k+1}{4}+\frac{1}{4}\sum_{1\le x\le k-1}\(\frac{x^2-kx}{p}\).$$
Noting that
\begin{align*}
r_{p-k}=\#\{1\le x\le k-1: \(\frac{x}{p}\)=1\}
=\#\{p+1-k\le x\le p-1: \(\frac{x}{p}\)=1\},
\end{align*}
we therefore have
$$r_k+r_{p-k}=\frac{p-1}{2}-\frac12(1+\(\frac{k}{p}\)).$$
Observing that
$$\sum_{1\le x\le k-1}\(\frac{x^2-kx}{p}\)=\sum_{p+1-k\le x\le p-1}\(\frac{x^2+kx}{p}\),$$
we obtain
\begin{align*}
\sum_{1\le x\le p-1-k}\(\frac{x^2+kx}{p}\)+\sum_{1\le x\le k-1}\(\frac{x^2-kx}{p}\)=&\sum_{1\le x\le p-1}\(\frac{x^2+kx}{p}\)\\
=&\sum_{0\le x\le p-1}\(\frac{(2x+k)^2-k^2}{p}\)=-1.
\end{align*}
The last equation follows from \cite[p.63 Exercise 8]{IR}.
In view of the above, we have
$$A_k^{++}+A_{p-k}^{++}=\frac{1}{4}(p-3-2\(\frac{k}{p}\)).$$
Thus
$$\prod_{1\le i<j\le n}(a_j-a_i)\equiv (-1)^{\sum_{1\le k\le n}A_{p-k}^{++}}\times
\prod_{1\le k\le n}k^{\frac{p-1}{4}}k^{\frac{-1}{2}(1+(\frac{k}{p}))}\pmod p.$$
Clearly
\begin{align*}
\sum_{1\le k\le n}A_{p-k}^{++}=&\sum_{1\le k\le n}\#\{1\le x\le k-1: \(\frac{x}{p}\)=\(\frac{k-x}{p}\)=1\}\\
=&\sum_{1\le k\le n}\#\{(x,y): 1\le x,y\le k-1,\ x+y=k,\ \(\frac{x}{p}\)=\(\frac{y}{p}\)=1\}=r_p^*,
\end{align*}
where $r_p^*$ is defined as in Theorem \ref{Thm p=1 mod4}.
We also have the following identities.
\begin{align*}
\prod_{1\le k\le n}k^{\frac{p-1}{4}}k^{\frac{-1}{2}(1+(\frac{k}{p}))}
=&\(\frac{p-1}{2}!\)^{\frac{p-1}{4}}\prod_{\substack{1\le k\le n\\ (\frac{k}{p})=1}}k^{-1}\\
=&\(\frac{p-1}{2}!\)^{\frac{p-5}{4}}\times s_p.
\end{align*}
Hence we get
\begin{equation}\label{equation denominator}
\prod_{1\le i<j\le n}(a_j-a_i)\equiv (-1)^{r_p^*}\cdot\(\frac{p-1}{2}!\)^{\frac{p-5}{4}}\cdot s_p\pmod p.
\end{equation}
Combining (\ref{equation numerator}) and (\ref{equation denominator}), our result follows from
Lemma \ref{Chowla lemma}--\ref{K.S Williams lemma}. \qed
\medskip

Now we concentrate on Theorem \ref{Thm primitive root}. Let $\zeta=e^{2\pi i/(p-1)}$ . We obtain the following lemma.

\begin{lemma}\label{p-1 th root of unity}
Let $p>3$ be a prime. We have the identity
$$\prod_{1\le i<j \le n}(\zeta^{2j}-\zeta^{2i})=n^{n/2}e^{(3n^2-n-2)\pi i/4}.$$
In particular, if $p\equiv 5\pmod8$, then the product is equal to $n^{\frac{n}{2}}(-1)^{\frac{n-2}{4}}\in\Z$.
\end{lemma}

\begin{proof}  Let $f(\zeta)=\prod_{1\le i<j\le n}(\zeta^{2j}-\zeta^{2i})$. Then
\begin{align*}
f(\zeta)^2=& (-1)^{n(n-1)/2}\cdot\prod_{1\le i \ne j \le n}(\zeta^{2j}-\zeta^{2i})\\
=& (-1)^{n(n-1)/2}\cdot\prod_{1\le j\le n}\prod_{i\ne j}(\zeta^{2j}-\zeta^{2i})\\
=& (-1)^{n(n-1)/2}\cdot\prod_{1\le j\le n}\frac{x^n-1}{x-\zeta^{2j}}|_{x=\zeta^{2j}}\\
=& (-1)^{n(n-1)/2}\cdot n^n\prod_{1\le j\le n}\zeta^{-2j}=n^n(-1)^{\frac{n^2+n+2}{2}}.
\end{align*}
Hence $\mid f(\zeta)\mid=n^{n/2}$. Next we consider the argument $\Arg(f(\zeta))$ of $f(\zeta)$. Since
$$\zeta^{2j}-\zeta^{2i}=\zeta^{i+j}(\zeta^{j-i}-\zeta^{i-j}),$$ we have
$$\Arg(\zeta^{2j}-\zeta^{2i})=\frac{2\pi(i+j)}{p-1}+\frac{\pi}{2}\sgn\(\sin(\frac{2\pi(j-i)}{p-1})\)=
\frac{2\pi}{p-1}(i+j)+\frac{\pi}{2}.$$
Hence
\begin{align*}
\Arg(f(\zeta))=&\sum_{1\le i<j\le n}\(\frac{2\pi}{p-1}(i+j)+\frac{\pi}{2}\)\\
=&\frac{n(n-1)\pi}{4}+\frac{2\pi}{p-1}\cdot\frac12\(\sum_{1\le i\le n}\sum_{1\le j\le n}(i+j)-\sum_{1\le i\le n}2i\)\\
\equiv & \frac{3n^2-n-2}{4}\pi\pmod {2\pi}.
\end{align*}
Thus $f(\zeta)=n^{n/2}e^{(3n^2-n-2)\pi i/4}$. If $p\equiv5\pmod8$, then $n\equiv 2\pmod4$.
Noting that $\frac{3n^2-n-2}{4}\in\Z$ and $\frac{3n^2-n-2}{4}\equiv \frac{n-2}{4}\pmod2$,
we have $f(\zeta)=n^{n/2}(-1)^{(n-2)/4}\in\Z.$
\end{proof}

\begin{remark}\label{Remark determinant}
Let $\chi$ and $g$ be as in Corollary \ref{determinant}. Let
$$ N_{p,g}=\left( \begin{matrix} \chi(g^2) & \chi(g^4) & \cdots & \chi(g^{p-1}) \\ \chi^2(g^2) & \chi^2(g^4) & \cdots &
\chi^2(g^{p-1}) \\ \vdots & \vdots & \ddots & \vdots \\ \chi^{\frac{p-1}{2}}(g^2) & \chi^{\frac{p-1}{2}}(g^4) & \cdots &
\chi^{\frac{p-1}{2}}(g^{p-1}) \\ \end{matrix} \right).$$
If $\chi(g)=\zeta$, then
$$\prod_{1\le i<j \le n}(\zeta^{2j}-\zeta^{2i})=-1\cdot\det(N_{p,g}).$$\
\end{remark}
\medskip

\noindent{\it Proof of Theorem}\ 1.2. Let $p\equiv 1\pmod 4$ be a prime, and let $g\in\Z$ be a primitive root modulo $p$. By definition we have
$$\sgn(\sigma_{0,2})=\prod_{1\le i<j\le n}\frac{\overline{j^2}-\overline{i^2}}{\overline{g^{2j}}-\overline{g^{2i}}}.$$
As in the proof of Theorem \ref{Thm p=1 mod4}, the numerator
\begin{equation}\label{equation numerator 2}
\prod_{1\le i<j\le n}(j^2-i^2)\equiv -n!\equiv -1\cdot\frac{(-1)^{(h(p)+1)/2}u_p}{2}\pmod p.
\end{equation}
We mainly focus on the denominator. We first observe the following fact.
Let $\Phi_{p-1}(x)$ be the $(p-1)$-th cyclotomic polynomial. Since $p>2$ and $p\equiv 1\pmod {p-1}$, it is known that
$p$ totally splits in $\Q(\zeta)$. Hence by Kummer's Theorem (cf. \cite[p.47 Proposition 8.3]{N}) we know that
$\Phi_{p-1}(x)\pmod {p\Z[x]}$ splits in $\Z/p\Z[x]$.
And the set of primitive $(p-1)$-th roots of unity of $\Q(\zeta)$ maps bijectively onto the set of primitive $(p-1)$-th roots of
unity of $\Z/p\Z$. Hence we have
$$\Phi_{p-1}(x)\equiv \prod_{\substack{1\le k\le p-1\\ \gcd(k,p-1)=1}}(x-g^k)\pmod {p\Z[x]}.$$
Now let $f(x)=\prod_{1\le i<j\le n}(x^{2j}-x^{2i})$. We may write
$$f(x)=g(x)\Phi_{p-1}(x)+h(x),$$
with $g(x),h(x)\in\Z[x]$ and $\deg(h(x))<\deg(\Phi_{p-1}(x))$. By Lemma \ref{p-1 th root of unity} we know that
$$h(\zeta)=n^{\frac{n}{2}}e^{(3n^2-n-2)\pi i/4}=n^{\frac{n}{2}}\zeta^{\frac{p-1}{4}\frac{3n^2-n-2}{2}}.$$
As $p\equiv1\pmod4$, we have $n^{\frac{n}{2}}\in\Z$.
Thus by Galois theory for each primitive $(p-1)$-th root of unity $\zeta_{p-1}$ we have
$$h(\zeta_{p-1})=n^{\frac{n}{2}}\cdot\zeta_{p-1}^{\frac{p-1}{4}\frac{3n^2-n-2}{2}}.$$

Let $$t(x)=n^{\frac{n}{2}}\cdot x^{\frac{p-1}{4}\frac{3n^2-n-2}{2}}\in\Z[x].$$
We obtain
$$\Phi_{p-1}(x)\mid h(x)-t(x).$$
Hence we have
\begin{equation}\label{equation denominator 2}
\prod_{1\le i<j\le n}(g^{2j}-g^{2i})=f(g)\equiv h(g)\equiv t(g)\equiv n^{\frac{n}{2}}\cdot g^{\frac{p-1}{4}\frac{3n^2-n-2}{2}}\pmod p.
\end{equation}
In particular, in view of the above, when $p\equiv 5\pmod 8$, for each primitive $(p-1)$-th root
of unity $\zeta_{p-1}$ we have
$$h(\zeta_{p-1})=n^{\frac{n}{2}}\cdot (-1)^{\frac{n-2}{4}}.$$
Hence $h(x)$ is the constant $n^{\frac{n}{2}}\cdot (-1)^{\frac{n-2}{4}}$.
The desired result follows from (\ref{equation numerator 2}) and (\ref{equation denominator 2}).\qed
\medskip

\noindent{\it Proof of Corollary}\ 1.1.
Let $M_p$ be as in Corollary \ref{determinant}, and let $N_{p,g}$ be as in Remark \ref{Remark determinant}.
When $p\equiv 5\pmod 8$, it is easy to see that $\det(N_{p,g})$ is independent on the choice of $g$.
Then
$$\det(M_p)=\sgn(\sigma_{0,2})\cdot\det(N_{p,g}).$$
Our desired result follows from Lemma \ref{p-1 th root of unity} and Theorem \ref{Thm primitive root}.\qed
\maketitle

\Ack\ We are exceedingly thankful for the careful reading and indispensable suggestions of the anonymous referees. We also thank Prof. Z.-W Sun and Prof. Hao Pan for their help suggestions.

This research was supported by the National Natural Science Foundation of
China (Grant No. 11571162).

\end{document}